\documentclass[pdftex]{amsart}
\pdfoutput=1
\usepackage{amssymb,amsmath,latexsym,amscd}
\usepackage{color}
\usepackage{graphicx}
\usepackage{bbm}
\usepackage[pdfpagelabels]{hyperref}
\usepackage{tikz,ifthen}

\newtheorem{theorem}{Theorem}[section]
\newtheorem{corollary}[theorem]{Corollary}

\newtheorem{lemma}[theorem]{Lemma}

\theoremstyle{definition}

\newtheorem{remark}[theorem]{Remark}

\newtheorem{example}[theorem]{Example}

\newcommand{\R}{\mathbbm{R}}
\newcommand{\C}{\mathbbm{C}}
\newcommand{\Z}{\mathbbm{Z}}
\newcommand{\N}{\mathbbm{N}}
\newcommand{\Q}{\mathbbm{Q}}
\renewcommand{\P}{\mathbbm{P}}
\newcommand{\tZ}{\tilde Z}
\newcommand{\cstar}{\C^*}
\newcommand{\torus}[1][n]{(\cstar)^{#1}}
\newcommand{\suchthat}{\ \colon \ }

\newcommand{\khov}{Khovansk\u{\i}i}

\DeclareMathOperator{\intr}{relint}
\DeclareMathOperator{\DMV}{DMV}        
\DeclareMathOperator{\ME}{ME}          
\DeclareMathOperator{\cayley}{Cayley}  
\newcommand{\hstar}{h^\star}
\DeclareMathOperator{\ehr}{ehr}        
\DeclareMathOperator{\supp}{supp}

\title{Discrete Mixed Volume and Hodge-Deligne Numbers}

\author[Di Rocco, Haase, Nill]{Sandra Di Rocco, Christian Haase, Benjamin Nill}

\address{Sandra Di Rocco \\ KTH Stockholm \\ Sweden}
\email{dirocco@math.kth.se}
\vspace{.1in}
\address{Christian Haase \\ Freie Universit\"at Berlin \\ Germany}
\email{haase@math.fu-berlin.de}
\vspace{.1in}
\address{Benjamin Nill \\ OvGU Magdeburg \\ Germany}
\email{benjamin.nill@ovgu.de}

\subjclass[2010]{}

\begin{document}

\begin{abstract}
Generalizing the famous Bernstein-Kushnirenko theorem, \khov\ 
proved in 1978 a combinatorial formula for the arithmetic genus of the
compactification of a generic complete intersection associated to a
family of lattice polytopes.
Recently, an analogous combinatorial formula, called the discrete
mixed volume, was introduced by Bihan and shown to be nonnegative. By
making a footnote of \khov\ in his paper explicit, we interpret
this invariant as the (motivic) arithmetic genus of the non-compact
generic complete intersection associated to the family of lattice
polytopes.
\end{abstract}

\maketitle

\section*{Introduction}
The by now classical Bernstein-Kushnirenko theorem (also called BKK theorem) is a gem linking
solutions of systems of polynomial equations to the combinatorics of
lattice polytopes~\cite{Bernstein}. It states that the number of
common solutions in $\torus$ to $n$ general equations is given by the
mixed volume of the $n$ associated Newton polytopes. The latter can be
computed as the alternating sum of the number of lattice points in
Minkowski sums of these $n$ lattice polytopes.

Recently, there has been significant progress on the combinatorial
side in the case of $k$ lattice polytopes $P_1, \ldots, P_k \subset \R^n$
where $k<n$. For a set of indices $I \subseteq [k] := \{1, \ldots,
k\}$ we write $P_I$ for the Minkowski sum \[P_I := \sum_{i \in I} P_i,\]
and $P_\emptyset := \{0\}$. Bihan~\cite{bihan-2014} calls the alternating sum
\begin{equation*}
  \DMV(P_1, \ldots, P_k) := \sum_{I \subseteq [k]} (-1)^{k-|I|} \ |P_I
  \cap \Z^n|
\end{equation*}
the {\em discrete mixed volume} of $P_1, \ldots, P_k$ and proves that
it is always non-negative. We remark that for $k=n$ this is precisely
the mixed volume. More generally, the polynomial
\[\ME(P_1, \ldots, P_k;m) := \DMV(mP_1, \ldots, mP_k)\] 
is called the {\em mixed Ehrhart
  polynomial}~\cite{HJST-mixed-ehrhart}. We recall that the {\em
  Ehrhart polynomial} of a lattice polytope $P \subset \R^n$ is given
by $\ehr(P)(m) := |m P \cap \Z^n|$ for $m \in \N$. New work by
Jochemko and Sanyal generalizes Bihan's positivity from counting
lattice points to more general valuations~\cite{DMVals}.

In this note, we clarify the algebraic geometric implications of these
combinatorial non-negativity results. We verify in
Theorem~\ref{main} that the discrete mixed volume $\ME(P_1, \ldots, P_k; 1)$ equals what we call, in the spirit of
\cite{Yokura}, the {\em motivic arithmetic genus} of a general complete
intersection in $\torus$ corresponding to $P_1, \ldots, P_k$.
This statement was already hinted at in a footnote of
Khovansk\u{\i}i~\cite[p.~41]{KhovanskiiGenus}. It is natural to ask, more generally, whether or not the motivic
arithmetic genus is non-negative for every smooth subvariety of the torus.

For smooth projective varieties,
this motivic arithmetic genus specializes to the usual arithmetic
genus (see Remark~\ref{rk:usual-genus}). In 1978 Khovansk\u{\i}i
\cite[Theorem~1]{KhovanskiiGenus} proved that in our setting the
arithmetic genus of a compactified smooth complete intersection equals
$\ME(P_1, \ldots, P_k;-1)$. In combination, these two results
may be seen as a motivic reciprocity theorem.

\smallskip 

This note is organized as follows. In Section~\ref{sec1} we recall the Hodge-Deligne polynomial, define the motivic arithmetic genus, and state our main result (Theorem~\ref{main}). Its proof is given in Section~\ref{sec2}.

\section{Genus formulae for complete intersections}
\label{sec1}

In this section, we set the notation, state our main result, and
present some open questions.

\subsection{Hodge-Deligne polynomials and the motivic arithmetic
  genus}\label{hodge-section}
  
Let us recall definition and properties of Hodge-Deligne
polynomials. We refer to \cite{DanilovKhovanskii} for more details.

Given a quasi-projective variety $Y$ over $\C$, the cohomology with
compact supports $H^k_c(Y,\Q) \otimes \C$ carries a natural mixed
Hodge structure~\cite{DeligneMixedHodge}. The dimension of the
$(p,q)$-piece is denoted by $h^{(p,q)}(H^k_c(Y))$ giving rise to the
$(p,q)$ Euler characteristic $$ e^{(p,q)}(Y) := \sum_k (-1)^k \
h^{(p,q)}(H^k_c(Y)) \,. $$
If $Y$ is smooth and projective, the Hodge structure is pure so that
we only have one summand, $e^{(p,q)}(Y) = (-1)^{p+q} \ h^{(p,q)}(Y)$,
the usual Hodge number.

The generating function for these numbers
$$E(Y; u,v) := \sum_{p,q} e^{p,q}(Y) u^p v^q \in \Z[u,v]$$
is the {\em Hodge-Deligne polynomial} (or {\em $E$-polynomial}). 

All we need to know about these polynomial invariants is that they
behave nicely under stratifications.

\begin{theorem} \label{thm:E-polynomial}
  The invariant $E$ factors through the Grothendieck ring, i.e., 
  \begin{enumerate}
  \item $E(\{ \text{point} \}) = 1$\,,
  \item if $X = X_1 \sqcup X_2$ with $X_i \subset X$ locally closed
    for $i=1,2$, then $E(X) = E(X_1)+E(X_2)$\,, and
  \item $E(X_1 \times X_2) = E(X_1) \cdot E(X_2)$\,.
  \end{enumerate}
\end{theorem}

It follows that $E$ behaves multiplicatively on fibrations.

\begin{corollary}
  If $\pi \colon Y \to X$ is Zariski-locally trivial with fiber $F$,
  then $E(Y) = E(X) \cdot E(F)$.
  \label{fiber}
\end{corollary}

\begin{example}{\rm 
Using these tools, we can compute the $E$-polynomial for toric
varieties $X(\Sigma)$ from the $f$-vector of the defining fan
$\Sigma$.
As $\P^1$ is smooth and projective with Betti numbers $(1,0,1)$, we
must have $E(\P^1) = uv+1$ whence $E(\C^* = \P^1 \setminus
\{0,\infty\}) = E(\P^1) -2 = uv-1$ so that $E((\C^*)^d) = (uv-1)^d$.
Using the stratification of $X(\Sigma)$ by tori, we get
$$E(X(\Sigma)) = \sum_d f_d(\Sigma) \ (uv-1)^{n-d} \,.$$
}
\end{example}

We define 
$$e^{p,+}(Y) := \sum_q e^{p,q}(Y) \in \Z$$
the {\em $p^{\rm th}$ $\chi_y$-characteristic}
. In particular, $e^{0,+}(Y)=E(Y;0,1)$. We denote
$$(-1)^{\dim(Y)} E(Y;0,1)$$ 
as the {\em motivic arithmetic genus} of a (not necessarily compact)
variety $Y$.

\begin{remark} \label{rk:usual-genus}
  In the traditional situation of a nonsingular
  projective variety $X$, we use the term {\em arithmetic genus}
  instead of motivic arithmetic genus. We remark that \khov\ uses in
  \cite{KhovanskiiGenus} the term `arithmetic genus' for $E(X; 0,1)$
  while it refers to $(-1)^{\dim(X)} (E(X;0,1) - 1)$ in 
  Hartshorne~\cite[III,Ex.5.3]{Hartshorne}. We prefer the above definition as it
  will fit nicely to the combinatorial notion. Observe that by
  using the birational invariance of the arithmetic genus, \khov\ 
  defines in \cite{KhovanskiiGenus} even the arithmetic genus of
  non-compact varieties as the arithmetic genus of some/any
  smooth projective compactification.
\end{remark}

\subsection{Our setup}
\label{notation}

Let $P_1, \ldots, P_k \subset \R^n$ be lattice polytopes, where we do
not impose any additional restrictions on their dimensions.
Let $f_1, \ldots, f_k \in \C[x_1,\ldots, x_n]$ be a generic $k$-tuple
of Laurent polynomials with Newton polytopes $P_1, \ldots, P_k$,
respectively. Then we denote by $Y$ the associated complete intersection in
$\torus$ of dimension $n-r$ defined by $f_1 = \cdots = f_k=0$. We
choose a compactification $\bar{Y}$ in a nonsingular projective toric
variety such that $\bar{Y}$ is nonsingular and intersects all torus
orbits transversally.

In the simplest example $k=1$, $n=2$, of a single polynomial in two
variables, $\bar Y$ will be a smooth curve of genus $g = | \intr{(P)}
\cap \Z^2 |$ so that $E(\bar Y) = uv - g(u+v) +1$ while $Y$ has
$b = | \partial P \cap \Z^2 |$ points removed, so that $E(Y) = uv -
g(u+v) - (b-1)$.

\subsection{The main result}

Let us recall \khov's formula for the arithmetic genus of
$\bar{Y}$. Here, $\intr(P)$ denotes the relative interior of a
polytope $P$. Recall that we agreed on $P_\emptyset = \{0\}$ so that
$|\intr(P_\emptyset) \cap \Z^n| = 1$. 

\begin{theorem}[\khov~\cite{KhovanskiiGenus}]
In the notation of Subsection~\ref{notation}, for $Y \not=\emptyset$,
$$(-1)^{\dim(Y)} E(\bar{Y};0,1) = \sum_{I \subseteq [k]}
(-1)^{\dim(P_I)-|I|} |\intr(P_I) \cap \Z^n|.$$
\label{kgenus}
\end{theorem}

Ehrhart-Macdonald reciprocity implies $\ehr(P_I)(-1) =
(-1)^{\dim(P_I)} |\intr(P_I) \cap \Z^n|$, see \cite{Beck}. Hence, 
\[(-1)^{\dim(Y)} E(\bar{Y};0,1) = \ME(P_1, \ldots, P_k;-1).\]

We remark that in contrast to the geometric genus the arithmetic genus
is not necessarily nonnegative. For instance, choose $P_1$ with
vertices $(0,0,0),(a,0,0),(0,a,0)$, and
$P_2$ with vertices $(0,0,1), (1,0,1)$. Then $\ME(P_1, P_2;-1) <\!\!<0$
for $a >\!\!> 0$.

\smallskip

As \khov\ indicated in the footnote on p.41 of
\cite{KhovanskiiGenus}, there is a corresponding result in the
non-compact situation. Here, for $\beta \in \Z_{\ge 0}^I$, we define
$|\beta| := \sum_{i\in I} \beta_i$. Moreover, we set $\Z^\emptyset :=
\{0\}$.

\begin{theorem}\label{main}
In the notation of Subsection~\ref{notation},
$$e^{p,+}(Y) = (-1)^{n-p}\sum_{I \subseteq [k]} (-1)^{|I|}
\left(\sum_{\substack{\beta \in \Z_{\ge 0}^I \\
      |\beta|\le p}} (-1)^{|\beta|} \binom{n+|I|}{p-|\beta|} |(P_I +
  P_\beta) \cap \Z^n|\right)$$ 
\end{theorem} 
In particular, we get for $p=0$
$$e^{0,+}(Y) = \sum_{I \subseteq [k]} (-1)^{n-|I|} |P_I \cap \Z^n| =
(-1)^{n-k}\DMV(P_1, \ldots, P_k)$$ 
In other words,
$$(-1)^{\dim(Y)} E(Y; 0,1) = \DMV(P_1, \ldots, P_k) = \ME(P_1, \ldots,
P_k;1)$$ 

\begin{corollary}
The motivic arithmetic genus of a generic complete intersection in the
algebraic torus associated to a family of lattice polytopes is
nonnegative. The generic complete intersection is non-empty if and
only if the motivic arithmetic genus is positive.
\end{corollary}

\begin{proof}
Nonnegativity of the discrete mixed volume is the central result in
\cite{bihan-2014}. It remains to prove the second statement. In
Theorem~3.17 of \cite{DMVals} it is shown that the discrete mixed
volume of $P_1, \ldots ,P_n$ is positive if and only if there are
linearly independent segments $S_1 \subseteq P_1, \ldots, S_k
\subseteq P_k$ with vertices in $\Z^k$. By the proof of Lemma~5.1.9 in
\cite{Schneider} this is equivalent to $P_1, \ldots, P_k$ satisfying
the `1-independence' condition in Definition~3.1 in
\cite{Batyrev-Borisov}. Theorem~3.3 in \cite{Batyrev-Borisov} yields
that this condition is equivalent to the nonemptiness of the complete
intersection.
\end{proof}

Our proof of Theorem~\ref{main} follows directly the ideas outlined in
\cite{DanilovKhovanskii}. In this fundamental paper, a formula for the
$\chi_y$-characteristic was given in the case of $Y$ being a
hypersurface, and an algorithm on how to generalize from hypersurfaces
to complete intersections was described. Let us also remark that a
complete formula for the Hodge-Deligne polynomial of $Y$ in the case
of a hypersurface was given in
\cite[Theorem~3.24]{Inventiones}.

\section{Proof of Theorem~\ref{main}}
\label{sec2}

\subsection{Ehrhart theory}
\label{ehrhart}

Given a lattice polytope $P \subset \R^n$, the {\em Ehrhart
  polynomial} of $P$ is given by $\ehr_P(k) := |kP \cap \R^n|$ for $k
\in \Z_{\ge 0}$. The {\em Ehrhart generating function} is of the form 
\[\sum_{j=0}^\infty \ehr(P;j) t^j = \frac{\sum_{k=0}^n
    \hstar_k(P)}{(1-t)^{\dim(P)+1}},\]
where $\hstar_0(P), \ldots, \hstar_d \in \Z_{\ge 0}$, and $\hstar_k =
0$ for $k > \dim(P)$. For $k=0, \ldots, n$, we have the following
relation
\begin{equation}
  \label{ehr2hstar}
  \hstar_k(P) = \sum_{j=0}^k (-1)^{k-j} \binom{\dim(P)+1}{k-j} \ehr(P;j)
\end{equation}
We say for lattice polytopes $P,Q \subset \R^n$ that $P$ is a {\em
  lattice pyramid} over $Q$ if there is an affine-linear
transformation of $\R^n$ bijectively mapping $\Z^n$ onto
$\Z^n$ such that $Q$ is mapped onto $Q' \subset \R^{n-1} \times \{0\}
\subset \R^n$ and $P$ is mapped onto the convex hull of $Q'$ and $(0,
\ldots, 0, 1)$. The $\hstar$-coefficients are invariant under lattice
pyramid constructions.

Let $P_1, \ldots, P_k$ be given as in Subsection~\ref{notation}. For
$I \subseteq [k]$ we define the {\em Cayley polytope}
\[C_I = \cayley(P_i \suchthat i \in I)\] 
as the lattice polytope in $\R^{n+k}$ with vertices $P_i \times
\{e_i\}$ for $i \in I$ (and $C_\emptyset := \emptyset$). Then $C_I$ is
a lattice polytope of dimension $\dim C_I = \dim(P_I) + |I| -1$. For
$\alpha \in \Z_{\ge 0}^k$ we define $|\alpha| := \alpha_1 + \cdots +
\alpha_k$ and $\supp(\alpha) := \{i \in [k] \,:\, \alpha_i
\not=0\}$. We define the Minkowski sum
\[ P_\alpha := \sum_{i \in I} \alpha_i P_i \,. \]
With this notation we can compute the Ehrhart polynomial of $C_I$
as follows.
\begin{equation} \label{ehrCayley}
  \ehr(C_{I};j) = \sum_{%
    \substack{\alpha \in \Z_{\ge 0}^I \\
      |\alpha|=j}}
  |P_\alpha \cap \Z^n| =
  \sum_{%
    \substack{\alpha \in \Z_{\ge 0}^k \\
      |\alpha|=j,\, \supp(\alpha) \subseteq I}}
  |P_\alpha \cap \Z^n| 
\end{equation}

\subsection{Some binomial identities}

Let us recall the following binomial identities (e.g., see
Table 169 in \cite{Knuth}):

\begin{equation}
\sum_{s \in \Z} \binom{a}{q+s} \binom{b}{w+s} = \binom{a+b}{a-q+w}
\label{binom1}
\end{equation}
for $a,b \in \Z_{\ge 0}$ and $q,w \in \Z$.

\begin{equation}
\sum_{s=0}^{\infty} (-1)^s \binom{a}{s} \binom{b+s}{q} = (-1)^{a} \binom{b}{q-a}
\label{binom2}
\end{equation}
for $a \in \Z_{\ge 0}$ and $b,q \in \Z$.

\subsection{A special case}

The following result is the key situation in the proof of Theorem~\ref{main}.

\begin{lemma}
\label{keylemma}
Let $P \subset \R^d$ be a lattice pyramid over a lattice polytope $\emptyset \not= Q
\subset \R^d$. We denote by $Z \subset (\cstar)^d$ the generic
hypersurface associated to $P$. Then
$$e^{p,+}(Z)=(-1)^{d-1-p} \left(\binom{d}{p+1} +\sum_{j=0}^\infty
  (-1)^{j+1} \binom{d}{d-p+j-1} \ehr(Q;j)\right).$$ 
\end{lemma}

\begin{proof}
Let $\dim(P)=d-c$, so $\dim(Q)=d-1-c$. 
As $Z$ is generic, up to isomorphism we can assume that $Z = Z' \times
(\cstar)^c$ for $Z' \subset (\cstar)^{d-c}$ generic hypersurface
associated to $P \subset \R^{d-c}$. Therefore, $E(Z;u,v) = E(Z';u,v)
(uv-1)^c$. This implies
\begin{equation}
\label{step}
e^{p,+}(Z) = \sum_{m \in \Z} e^{m,+}(Z') (-1)^{m+c-p} \binom{c}{p-m}
\end{equation}
By \cite[4.6]{DanilovKhovanskii} we have for $m \ge 0$
$$(-1)^{d-c-1} e^{m,+}(Z') = (-1)^m \binom{d-c}{m+1} + \hstar_{m+1}(P).$$
Note that this equation also holds for $m < 0$.
Plugging this into \eqref{step} yields
$$e^{p,+}(Z)=\sum_{m \in \Z} \left[ (-1)^{d-c-1} \left( (-1)^m \binom{d-c}{m+1}
    + \hstar_{m+1}(P) \right) \right] (-1)^{m+c-p}\binom{c}{p-m}$$
$$= (-1)^{d-1-p} \left[\left(\sum_{m\in\Z}
    \binom{d-c}{m+1}\binom{c}{c-p+m}\right) + \sum_{m\in\Z} (-1)^m
  \binom{c}{p-m} \hstar_{m+1}(P)\right]$$

Binomial identity \eqref{binom1} shows that the expression in the
round parentheses evaluates to $\binom{d}{d-p-1} = \binom{d}{p+1}$. As
$\hstar_{m+1}(P)=\hstar_{m+1}(Q)$, by equation \eqref{ehr2hstar} the
previous expression equals
$$(-1)^{d-1-p} \left(\binom{d}{p+1} + \sum_{m\in\Z} (-1)^m
  \binom{c}{p-m} \sum_{j=0}^\infty (-1)^{m+1-j} \binom{d-c}{m+1-j}
  \ehr(Q;j)\right)$$
$$=(-1)^{d-1-p} \left[\binom{d}{p+1} + \sum_{j=0}^\infty (-1)^{j+1}
  \ehr(Q;j) \left(\sum_{m\in\Z} \binom{c}{c-p+m}
    \binom{d-c}{1-j+m}\right)\right]$$
Binomial identity \eqref{binom1} shows that the sum in the round
parentheses evaluates to $\binom{d}{p+1-j} = \binom{d}{d-p+j-1}$. This
finishes the proof.
\end{proof}

\subsection{Proof of Theorem~\ref{main}}

We follow the procedure described in \cite[6.2]{DanilovKhovanskii} on
how to reduce the computation of the Hodge-Deligne polynomials from
the complete intersection case to that of a hypersurface (also called
the Cayley trick).
Let $\tZ \subset \torus \times \C^{k}$ be the hypersurface of $F = 1+
\sum y_i f_i \in \C[x_1, \ldots, x_n, y_1, \ldots, y_k]$.
As $\tZ \to \torus \setminus Y$ is a bundle with fiber $\C^{k-1}$ over
$\torus \setminus Y$, we get from Corollary~\ref{fiber}
 $$ E(\tZ;u,v) = (uv)^{k-1} \left[ (uv-1)^n \ - \ E(Y;u,v) \right]$$
Therefore,
$$E(Y;u,v) = (uv-1)^n - \frac{1}{(uv)^{k-1}} E(\tZ;u,v)$$
This yields,
$$e^{p,+}(Y) = (-1)^{n-p} \binom{n}{p} - \sum_{q=0}^\infty e^{p+k-1,q+k-1}(\tZ).$$

As $(uv)^{k-1} | E(\tZ;u,v)$, we have $e^{p+k-1,q'}(\tZ)=0$ for $q' <
k-1$, thus
\begin{equation}
\label{eqeins}
e^{p,+}(Y) = (-1)^{n-p} \binom{n}{p} - e^{p+k-1,+}(\tZ).
\end{equation}

For $I \subseteq [k]$ let us define the strata
$$Z_I := \tZ \cap \{ y_j \neq 0 \suchthat j \in I \} \cap \{ y_j=0
\suchthat j \not\in I \}$$

Because $\tZ = \bigsqcup_I Z_I$ (and $Z_\emptyset = \emptyset$) we get 
\begin{equation}\label{eqzwei}
E(\tZ;u,v) = \sum_{\emptyset \not= I \subseteq [k]} E(Z_I;u,v).
\end{equation}
For $I \not=\emptyset$, by construction, $Z_I$ is the generic hypersurface in
$(\cstar)^{n+|I|}$ associated to a lattice pyramid over the Cayley
polytope $C_I \not=\emptyset$.
Now, we are in the special case of Lemma~\ref{keylemma}. Hence, we get
for $I \not=\emptyset$

$$e^{p+k-1,+}(Z_I) = (-1)^{n+|I|-p-k}
\left(\binom{n+|I|}{p+k}+\sum_{j=0}^\infty(-1)^{j+1}
  \binom{n+|I|}{n+|I|-p-k+j} \ehr(C_I;j)\right)$$

Therefore, from this equation together with \eqref{eqeins} and
\eqref{eqzwei}, we see that
\begin{equation}
\label{main-equation}
e^{p,+}(Y) = (-1)^{n-p} \binom{n}{p} + (-1)^{n+1-p-k}
\left(\textit{expression}_1 - \textit{expression}_2\right),
\end{equation}
where 
\begin{itemize}
\item $\textit{expression}_1$ equals 
$$\sum_{\emptyset \not= I \subseteq [k]} (-1)^{|I|} \binom{n+|I|}{p+k}$$
\item $\textit{expression}_2$ equals 
$$ \sum_{\emptyset \not= I \subseteq [k]} (-1)^{|I|} \sum_{j=0}^\infty
(-1)^{j} \binom{n+|I|}{p+k-j} \ehr(C_I;j)$$
\end{itemize}
Here, $\textit{expression}_1$ can be rewritten as
$$\left(\sum_{s=0}^\infty (-1)^s \binom{k}{s} \binom{n+s}{p+k}\right)
- \binom{n}{p+k}$$
which gets simplified by binomial expression \eqref{binom2} to 
\begin{equation}
\label{expr1-solved}
(-1)^k \binom{n}{p} - \binom{n}{p+k}
\end{equation}
It remains to consider $\textit{expression}_2$. By \eqref{ehrCayley}
it evaluates to $$\sum_{j=0}^\infty (-1)^{j} \sum_{\emptyset \not= I
  \subseteq [k]} (-1)^{|I|} \binom{n+|I|}{p+k-j} \sum_{|\alpha| = j,\;
  \supp \alpha \,\subseteq\, I}
\ehr(P_\alpha;1)$$ $$=\left(\sum_{\alpha \in \Z_{\ge0}^k}
  \ehr(P_\alpha; 1) (-1)^{|\alpha|} \sum_{[k] \supseteq I \supseteq
    \supp \alpha} (-1)^{|I|} \binom{n+|I|}{p+k-|\alpha|}\right) -
\binom{n}{p+k} $$
Let us define $t_\alpha := |\supp \alpha|$. By introducing the
counting variable $s = |I \setminus \supp \alpha|$, the previous
expression becomes
$$\left(\sum_{\alpha \in \Z_{\ge0}^k} \ehr(P_\alpha; 1) (-1)^{|\alpha|+t_\alpha}
\left[\sum_{s=0}^{k-t_\alpha} (-1)^s \binom{k-t_\alpha}{s}
  \binom{n+t_\alpha+s}{p+k-|\alpha|}\right] \right) - \binom{n}{p+k}$$
Binomial identity \eqref{binom2} shows that the sum in the brackets
evaluates to $(-1)^{k-t_\alpha}
\binom{n+t_\alpha}{p+t_\alpha-|\alpha|}$. Hence, we can rewrite the
previous expression as
$$\left(\sum_{\alpha \in \Z_{\ge0}^k}  (-1)^{k+|\alpha|}
\binom{n+t_\alpha}{p+t_\alpha-|\alpha|} \ehr(P_\alpha; 1)\right) - \binom{n}{p+k}$$
Writing $\alpha$ as the characteristic vector of its support plus a
vector $\beta$ reformulates $\textit{expression}_2$ as
\begin{equation}
\label{expr2-solved}
\left(\sum_{I \subseteq [k]} (-1)^{k-|I|} \sum_{\substack{\beta \in \Z_{\ge 0}^I}} (-1)^{|\beta|} \binom{n+|I|}{p-|\beta|} \ehr(P_I
  + P_\beta; 1)\right) - \binom{n}{p+k}
\end{equation}
Finally, plugging \eqref{expr1-solved} and \eqref{expr2-solved} into
\eqref{main-equation} and some simplification yields
$$e^{p,+}(Y) =(-1)^{n-p}\sum_{I \subseteq [k]} (-1)^{|I|} \sum_{\substack{\beta \in
    \Z_{\ge 0}^I}} (-1)^{|\beta|} \binom{n+|I|}{p-|\beta|} \ehr(P_I +
P_\beta; 1)$$

\hfill$\qed$

\subsubsection*{Acknowledgements}
This work started at August 24-26, 2015 at KTH and SU and was finished
at the Fields Institute Toronto during the program on Combinatorial
Algebraic Geometry. The authors would like to thank these institutions
for hospitality and financial support. BN is an affiliated researcher
with Stockholm University and partially supported by the
Vetenskapsr{\aa}det grant~NT:2014-3991.

\bibliographystyle{alpha}
\bibliography{hodge}

\end{document}